\documentclass[a4paper,reqno,11pt]{amsart}

\usepackage[T1]{fontenc}
\usepackage{times}

\usepackage{amsmath}
\usepackage{amssymb}
\usepackage{amsthm}
\usepackage{amsfonts}

\usepackage{mathtools}
\usepackage{constants}
\usepackage{enumerate}
\usepackage[hidelinks]{hyperref}
\usepackage[nameinlink]{cleveref}
\usepackage[backrefs]{amsrefs}

\usepackage{empheq}

\numberwithin{equation}{section}

\newtheorem{theorem}{Theorem}
\newtheorem{proposition}{Proposition}[section]
\newtheorem{lemma}[proposition]{Lemma}

\theoremstyle{remark}
\newtheorem{remark}[proposition]{Remark}
\newtheorem{step}{Step}

\makeatletter
\@addtoreset{step}{section}
\makeatother

\newcommand{\R}{\mathbb{R}}
\newcommand{\N}{\mathbb{N}}

\newcommand{\loc}{\mathrm{loc}}

\newcommand{\DL}[1]{(\Delta, L^{#1})}

\newcommand{\cH}{\mathcal{H}}

\newcommand{\cM}{\mathcal{M}}

\newcommand{\defeq}{\coloneqq}

\DeclarePairedDelimiter{\abs}{\lvert}{\rvert}
\DeclarePairedDelimiter{\bracks}{\lbrack}{\rbrack}
\DeclarePairedDelimiter{\meas}{\lvert}{\rvert}
\DeclarePairedDelimiter{\norm}{\lVert}{\rVert}
\DeclarePairedDelimiter{\paren}{\lparen}{\rparen}
\DeclarePairedDelimiter{\set}{\lbrace}{\rbrace}

\DeclareMathOperator{\capt}{cap}

\DeclareMathOperator{\supp}{supp}

\newcommand{\NP}{N\!}

\let\d\relax
\newcommand{\d}{\mathop{}\!\mathrm{d}}

\let\c\relax
\newcommand{\c}{\mathrm{c}}

\renewconstantfamily{normal}{symbol=C}

\title[Optimal control of elliptic problems with sparsity]{Optimal control of nonlinear elliptic problems with sparsity}

\author{Augusto C. Ponce}
\address{
Augusto C. Ponce\hfill\break\indent
Universit{\'e} catholique de Louvain\hfill\break\indent
Institut de Recherche en Math{\'e}matique et Physique\hfill\break\indent
Chemin du Cyclotron 2, bte L7.01.02\hfill\break\indent
1348 Louvain-la-Neuve\hfill\break\indent
Belgium}
\email{Augusto.Ponce@uclouvain.be}

\author{Nicolas Wilmet}
\address{
Nicolas Wilmet\hfill\break\indent
Universit{\'e} catholique de Louvain\hfill\break\indent
Institut de Recherche en Math{\'e}matique et Physique\hfill\break\indent
Chemin du Cyclotron 2, bte L7.01.02\hfill\break\indent
1348 Louvain-la-Neuve\hfill\break\indent
Belgium}
\email{Nicolas.Wilmet@uclouvain.be}

\date{\today}



\begin{document}

\begin{abstract}
We study the minimization of the cost functional
\[
F(\mu) = \norm{u - u_d}_{L^p(\Omega)} + \alpha \norm{\mu}_{\cM(\Omega)},
\]
where the controls \(\mu\) are taken in the space of finite Borel measures and \(u \in W_0^{1, 1}(\Omega)\) satisfies the equation \(- \Delta u + g(u) = \mu\) in the sense of distributions in \(\Omega\) for a given nondecreasing continuous function \(g : \R \to \R\) such that \(g(0) = 0\). We prove that \(F\) has a minimizer for every desired state \(u_d \in L^1(\Omega)\) and every control parameter \(\alpha > 0\). We then show that when \(u_d\) is nonnegative or bounded, every minimizer of \(F\) has the same property.
\end{abstract}

\maketitle
%\tableofcontents

\section{Introduction and main results}

Let \(N \ge 2\) and \(\Omega \subset \R^N\) be a smooth bounded open set. Inspired by recent works of Casas, Clason and Kunisch~\citelist{\cite{ClasonKunisch:2011} \cite{CasasClasonKunisch:2012} \cite{CasasKunisch:2014}}, we investigate an optimal control problem with sparsity involving the nonlinear problem
\begin{equation}
\label{eq:nonlinearDirichletProblem}
\left\{
\begin{alignedat}{2}
-\Delta u + g(u) &= \mu && \quad \text{in \(\Omega\),} \\
u &= 0 && \quad \text{on \(\partial\Omega\),}
\end{alignedat}
\right.
\end{equation}
where \(g : \R \to \R\) is a nondecreasing continuous function such that \(g(0) = 0\). More precisely, we consider the minimization of the cost functional \(F : L^1(\Omega) \to \bracks{0, \infty}\), defined for \(\mu \in L^1(\Omega)\) by
\[
F(\mu) = \norm{u - u_d}_{L^p(\Omega)} + \alpha \norm{\mu}_{L^1(\Omega)},
\]
where \(u\) is the unique solution of \eqref{eq:nonlinearDirichletProblem} corresponding to \(\mu\). Here, \(p\) is any exponent satisfying \(1 \le p \le \infty\), \(u_d \in L^1(\Omega)\) is the given ideal (or desired) state and \(\alpha > 0\) is the control parameter.

Stadler observed in \cite{Stadler:2009} that---in contrast with the usual Hilbert-space \(L^2\) setting---the use of the \(L^1\) norm in the cost functional leads in many cases to optimal controls which are concentrated in a small region of the domain (sparsity phenomenon). This property is, for example, relevant in determining the optimal placement of actuators in distributed parameter systems, where control devices cannot be put all over the domain; see \cite{Casas:2017} for a recent review on sparsity in optimal control of partial differential equations.

When \(\mu \in L^2(\Omega)\), the solution of \eqref{eq:nonlinearDirichletProblem} belongs to \(W_0^{1, 2}(\Omega)\) and can be obtained by standard minimization of the associated energy functional; see e.g.~\cite{Ponce:2016}. The existence of a solution of \eqref{eq:nonlinearDirichletProblem} for any \(\mu \in L^1(\Omega)\) is due to Brezis and Strauss~\cite{BrezisStrauss:1973} and follows from approximation of \(\mu\) by a sequence of \(L^2\) functions. In this case, the unique solution of \eqref{eq:nonlinearDirichletProblem} is a function \(u \in W_0^{1, 1}(\Omega)\) such that \(g(u) \in L^1(\Omega)\) for which the equation
\[
- \Delta u + g(u) = \mu \quad \text{in \(\Omega\)}
\]
holds in the sense of distributions.

Due to the lack of weak compactness, the space \(L^1(\Omega)\) is not suitable in discussing the existence of minimizers of \(F\). It is therefore natural to enlarge the minimization class to the entire family of finite Borel measures, thus recovering weak compactness. Another advantage of such an extension is that controls \(\mu\) concentrated on small sets with Lebesgue measure zero are also allowed.

In the sequel, we denote by \(\cM(\Omega)\) the Banach space of finite Borel measures on \(\Omega\) equipped with the \emph{total variation} norm
\[
\norm{\mu}_{\cM(\Omega)} = \abs{\mu}(\Omega).
\]
Since the \(L^1\) norm and the total variation norm coincide on \(L^1(\Omega)\), we extend the functional \(F\) to \(\cM(\Omega)\) by setting
\[
F(\mu) = \norm{u - u_d}_{L^p(\Omega)} + \alpha \norm{\mu}_{\cM(\Omega)}
\]
when the Dirichlet problem \eqref{eq:nonlinearDirichletProblem} with datum \(\mu \in \cM(\Omega)\) is solvable. When \eqref{eq:nonlinearDirichletProblem} does not have a solution for a certain \(\mu \in \cM(\Omega)\), we let \(F(\mu) = \infty\).

The optimal control problem considered in this paper is henceforth the following:
\begin{equation}
\label{eq:optimalControlProblem}
\boxed{\hspace{.5em}\text{To find \(\mu^* \in \cM(\Omega)\) such that \(F(\mu^*) = \inf^{\phantom{}}_{\mu \in \cM(\Omega)} {} F(\mu)\).}\hspace{.5em}}
\end{equation}
This minimization problem is only relevant when \(F \not\equiv \infty\). Throughout the paper, we thus restrict our attention to ideal states \(u_d\) such that
\begin{equation}
\label{eq:finiteCostFunctional}
F \not\equiv \infty \quad \text{in \(\cM(\Omega)\).}
\end{equation}
This happens for example when \(u_d \in L^p(\Omega)\) since \(F(0) = \norm{u_d}_{L^p(\Omega)}\), but such an assumption is not necessary for \eqref{eq:finiteCostFunctional} to hold. For instance, if \(\frac{N}{N - 2} \le p \le \infty\) and the nonlinearity \(g\) satisfies the growth assumption
\begin{equation}
\label{eq:subcriticalGrowth}
\abs{g(t)} \le C (\abs{t}^q + 1),
\end{equation}
where \(1 \le q < p\), then there exist ideal states \(u_d \in L^1(\Omega)\) with \(u_d \not\in L^p(\Omega)\) satisfying a Lavrentiev phenomenon:
\[
\inf_{\mu \in \cM(\Omega)} {} F(\mu) < \inf_{\mu \in L^1(\Omega)} {} F(\mu) = \infty;
\]
see~\Cref{prop:nonexistenceSummableControl,prop:nonexistenceSummableControlInf} below.

One of our main results concerning an arbitrary nondecreasing continuous function \(g\) is
\begin{theorem}
\label{thm:existence}
The optimal control problem \eqref{eq:optimalControlProblem} has a solution for every \(1 \le p \le \infty\), \(u_d \in L^1(\Omega)\) and \(\alpha > 0\).
\end{theorem}

Problem \eqref{eq:optimalControlProblem} has been studied in \cite{CasasKunisch:2014} for nonlinearities \(g\) which are subcritical in the sense that \eqref{eq:subcriticalGrowth} is satisfied for some exponent \(1 \le q < \frac{N}{N-2}\). For such a nonlinearity \(g\), B\'enilan and Brezis~\cite{BenilanBrezis:2003} proved that \eqref{eq:nonlinearDirichletProblem} is solvable for every \(\mu \in \cM(\Omega)\), and solutions of \eqref{eq:nonlinearDirichletProblem} are compactly embedded in \(L^q(\Omega)\); see also~\cite{Ponce:2016}*{Proposition~21.1}. One then deduces the lower semicontinuity of \(F\) with respect to weak* convergence in \(\cM(\Omega)\).

In the supercritical case \(q \ge \frac{N}{N-2}\) for \(N \ge 3\), the functional \(F\) is no longer lower semicontinuous with respect to weak* convergence in \(\cM(\Omega)\) when \(1 \le p \le q\). The reason is that if \((\mu_n)_{n \in \N}\) is a sequence in \(\cM(\Omega)\) converging weakly* to \(\mu\) in \(\cM(\Omega)\) and if the sequence \((u_n)_{n \in \N}\) of solutions of \eqref{eq:nonlinearDirichletProblem} corresponding to \(\mu_n\) converges strongly to some function \(u\) in \(L^1(\Omega)\), then \(u\) need not be a solution of \eqref{eq:nonlinearDirichletProblem} involving \(\mu\). For example---in the spirit of \cite{BenilanBrezis:2003}*{Remark~A.4}---take \(\Omega = B(0; 1)\) the unit ball in \(\R^N\), \(g(t) = \abs{t}^{p-1}t\) and a sequence of mollifiers \(\mu_n = \rho_n\). In this case, the sequence \((\rho_n)_{n \in \N_*}\) converges weakly* to the Dirac mass \(\delta_0\), but one has \(F(\delta_0) = \infty\) while \(\limsup\limits_{n \to \infty} {} F(\mu_n) < \infty\); see~\Cref{prop:notLowerSemicontinuous} below.

To handle \Cref{thm:existence} in full generality, we propose a different approach based on a lower semicontinuity property of the \emph{reduced limit} that we establish in \Cref{sec:RL} below. The concept of reduced limit associated to sequences of measures in connection with \eqref{eq:nonlinearDirichletProblem} has been introduced in \cite{MarcusPonce:2010}.

In what follows, let \(\mu^*\) be any solution of the minimization problem \eqref{eq:optimalControlProblem}. In optimal control theory, such a solution is called an optimal control. The optimal state \(u^*\) associated to \(\mu^*\) is the unique solution of \eqref{eq:nonlinearDirichletProblem} corresponding to \(\mu^*\). We show that \(u^*\) shares many properties with \(u_d\). For example, one has

\begin{theorem}
\label{thm:nonnegative}
If \(u_d\) is nonnegative, then \(u^*\) is also nonnegative.
\end{theorem}

Similarly, if \(u_d\) is bounded, then the same is true for \(u^*\). More precisely,

\begin{theorem}
\label{thm:bounded}
If \(u_d \in L^\infty(\Omega)\), then \(u^* \in L^\infty(\Omega)\) and
\[
\norm{u^*}_{L^\infty(\Omega)} \le \norm{u_d}_{L^\infty(\Omega)}.
\]
\end{theorem}

%\cite{Ponce:2016}*{Lemma~5.8}

By standard interpolation, one deduces from \Cref{thm:bounded} that \(u^* \in W_0^{1, 2}(\Omega)\); see~\cite{CasasKunisch:2014}*{Theorem~5.1} in the subcritical case.

The paper is organized as follows. In \Cref{sec:RL}, we recall the notion of reduced limit from~\cite{MarcusPonce:2010} and we prove a new property concerning the lower semicontinuity of the total variation norm with respect to the reduced limit. \Cref{thm:existence} is then proved in \Cref{sec:existence}. In \Cref{sec:truncation}, we show that given a solution \(u\) of \eqref{eq:nonlinearDirichletProblem} corresponding to \(\mu\), its truncation \(\min {} \set{u, w}\) with a nonnegative supersolution \(w\) yields a finite Borel measure whose total variation is \(\le \norm{\mu}_{\cM(\Omega)}\). We use this property in \Cref{sec:regularity} to prove \Cref{thm:nonnegative,thm:bounded}. \Cref{sec:regularization} is dedicated to the regularization of the desired state \(u_d\) and the stability of the minimization problem \eqref{eq:optimalControlProblem}. In \Cref{sec:idealStateMeasure,sec:nonexistenceSummableControl}, we show that \(\mu^*\) need not be a summable function when \(u_d\) is unbounded. In \Cref{sec:example}, we discuss the lack of convexity and lower semicontinuity of \(F\) with respect to weak* convergence in \(\cM(\Omega)\). In \Cref{sec:lavrentiev}, we show that the Lavrentiev phenomenon cannot occur if \(g = \abs{t}^{p - 1} t\) and \(u_d \in L^p(\Omega)\).

%We do not know whether the same is true if \(u_d\) is bounded.

\section{Reduced limits for nonlinear equations with measures}
\label{sec:RL}

In this section, we recall the notion of reduced limit introduced in~\cite{MarcusPonce:2010}. We also prove a new property of the reduced limit which is central in our proof of \Cref{thm:existence}; see~\Cref{prop:lowerSemicontinuityRL}. To motivate the concept of reduced limit, let us consider a bounded sequence \((\mu_n)_{n \in \N}\) of finite measures on \(\Omega\). By the weak compactness property of bounded sequences of measures, we may extract from \((\mu_n)_{n \in \N}\) a subsequence \((\mu_{n_k})_{k \in \N}\) converging weakly* to some finite measure \(\mu\) in \(\cM(\Omega)\)---that is,
\[
\lim_{k \to \infty} {} \int_\Omega {} \phi \d\mu_{n_k} = \int_\Omega {} \phi \d\mu,
\]
for every continuous function \(\phi : \overline{\Omega} \to \R\) with \(\phi = 0\) on \(\partial\Omega\); see e.g.~\cite{Ponce:2016}*{Proposition~2.8}. To simplify the notation, we assume in the following that the whole sequence \((\mu_n)_{n \in \N}\) converges weakly* to \(\mu\) in \(\cM(\Omega)\).

We now suppose that \eqref{eq:nonlinearDirichletProblem} with datum \(\mu_n\) has a solution for each \(n \in \N\) and we denote this unique solution by \(u_n\). In the literature, measures for which \eqref{eq:nonlinearDirichletProblem} has a solution are referred to as \emph{good measures}. One has the following estimate:
\begin{equation}
\label{eq:W011Estimate}
\norm{u_n}_{W_0^{1, 1}(\Omega)} \le C \norm{\mu_n}_{\cM(\Omega)},
\end{equation}
for some constant \(C > 0\) depending on \(N\) and \(\Omega\). This estimate can be deduced from elliptic estimates due to Littman, Stampacchia and Weinberger~\cite{LittmanStampacchiaWeinberger:1963}*{Theorem~5.1} and the following absorption estimate for solutions of \eqref{eq:nonlinearDirichletProblem} with measure data~\cite{BrezisMarcusPonce:2007}*{Proposition~4.B.3}:
\begin{equation}
\label{eq:absorptionEstimate}
\norm{g(u)}_{L^1(\Omega)} \le \norm{\mu}_{\cM(\Omega)}.
\end{equation}
By virtue of \eqref{eq:W011Estimate} and the Rellich--Kondrashov compactness theorem, taking a subsequence if necessary, we may thus assume that \((u_n)_{n \in \N}\) converges strongly in \(L^1(\Omega)\) to some function \(u\). In general, \(u\) is \emph{not} a solution of \eqref{eq:nonlinearDirichletProblem} involving \(\mu\); see e.g.~\cite{BrezisMarcusPonce:2007}*{Example~4.1}.

% \cite{Ponce:2016}*{Proposition~21.5} (absorption estimate)

The situation is, however, not as dramatic as it seems:

\begin{proposition}
\label{prop:existenceRL}
The sequence \((u_n)_{n \in \N}\) converges strongly in \(L^1(\Omega)\) to some function \(u^\#\) such that \(g(u^\#) \in L^1(\Omega)\) and there exists \(\mu^\# \in \cM(\Omega)\) such that \(u^\#\) is the unique solution of \eqref{eq:nonlinearDirichletProblem} corresponding to \(\mu^\#\).
\end{proposition}

In this case, we say that \(\mu^\#\) is the reduced limit of \((\mu_n)_{n \in \N}\). This is Theorem~1.1 in \cite{MarcusPonce:2010} regarding the existence of the reduced limit, whose proof is rather straightforward. A striking fact---much more difficult to prove---is that the reduced limit does not depend on the Dirichlet boundary condition; see~\cite{MarcusPonce:2010}*{Theorem~1.2}. In the sequel, whenever a sequence of finite measures \((\mu_n)_{n \in \N}\) is said to have a reduced limit, it is implicitly assumed that each measure \(\mu_n\) is a good measure.

Given a sequence \((\mu_n)_{n \in \N}\) of nonnegative measures in \(\cM(\Omega)\) with weak* limit \(\mu\) in \(\cM(\Omega)\) and reduced limit \(\mu^\#\), a straightforward application of Fatou's lemma gives the estimate
\[
\mu^\# \le \mu.
\]
A deeper property actually holds: if \((\mu_n)_{n \in \N}\) and \((\nu_n)_{n \in \N}\) are sequences in \(\cM(\Omega)\) with reduced limits \(\mu^\#\) and \(\nu^\#\), respectively, and if for every \(n \in \N\),
\[
\nu_n \le \mu_n,
\]
then
\begin{equation}
\label{eq:monotonicityRL}
\nu^\# \le \mu^\#;
\end{equation}
see~\cite{MarcusPonce:2010}*{Theorem~7.1}. In particular, if every measure \(\mu_n\) is nonnegative, then the reduced limit \(\mu^\#\) is also nonnegative.

The inequalities appearing above are meant in the sense of measures: given two finite measures \(\mu\) and \(\nu\) on \(\Omega\), one has
\begin{equation}
\label{eq:inequalityMeasures}
\nu \le \mu
\end{equation}
in the sense of measures if, for every Borel set \(A \subset \Omega\), \(\nu(A) \le \mu(A)\). In fact, inequality~\eqref{eq:inequalityMeasures} also holds in the sense of distributions in \(\Omega\):
\[
\int_\Omega {} \varphi \d\nu \le \int_\Omega {} \varphi \d\mu,
\]
for every nonnegative function \(\varphi \in C_c^\infty(\Omega)\); see \cite{Ponce:2016}*{Proposition~6.12} for the equivalence between the two notions.

We now prove the lower semicontinuity of the total variation norm with respect to the reduced limit. Once this new property is established, \Cref{thm:existence} can be proved along the lines of the direct method of the calculus of variations.

\begin{proposition}
\label{prop:lowerSemicontinuityRL}
For every bounded sequence \((\mu_n)_{n \in \N}\) in \(\cM(\Omega)\) with reduced limit \(\mu^\#\), we have
\[
\norm{\mu^\#}_{\cM(\Omega)} \le \liminf_{n \to \infty} {} \norm{\mu_n}_{\cM(\Omega)}.
\]
\end{proposition}

For the proof of \Cref{prop:lowerSemicontinuityRL}, we rely on the following property of good measures: if \(\mu \in \cM(\Omega)\) is a good measure, then \(\mu^+\) and \(- \mu^-\) are also good measures; see~\cite{BrezisMarcusPonce:2007}*{Theorem~4.9}. Here, \(\mu^+\) and \(\mu^-\) are the unique nonnegative measures given by the Jordan decomposition theorem such that
\[
\mu = \mu^+ - \mu^-
\]
and
\begin{equation}
\label{eq:normDecomposition}
\norm{\mu}_{\cM(\Omega)} = \norm{\mu^+}_{\cM(\Omega)} + \norm{\mu^-}_{\cM(\Omega)}.
\end{equation}

\begin{proof}[Proof of \Cref{prop:lowerSemicontinuityRL}]
By definition, the sequence \((\mu_n)_{n \in \N}\) converges weakly* to some finite measure \(\mu\) in \(\cM(\Omega)\). Taking a subsequence if necessary, we may assume that
\[
\lim_{n \to \infty} {} \norm{\mu_n}_{\cM(\Omega)} = \liminf_{n \to \infty} {} \norm{\mu_n}_{\cM(\Omega)}.
\]
In particular, the value of the limit does not change by taking a further subsequence of \((\mu_n)_{n \in \N}\). By the property of good measures mentioned above, we have that \(\mu_n^+\) and \(- \mu_n^-\) are good measures for every \(n \in \N\). Since the sequence \((\mu_n)_{n \in \N}\) is bounded in \(\cM(\Omega)\), one deduces from \eqref{eq:normDecomposition} that the sequences \((\mu_n^+)_{n \in \N}\) and \((- \mu_n^-)_{n \in \N}\) are also bounded in \(\cM(\Omega)\). Passing to a subsequence if necessary, we may thus assume that \((\mu_n^+)_{n \in \N}\) and \((- \mu_n^-)_{n \in \N}\) have weak* limits \(\mu_\oplus\) and \(\mu_\ominus\) in \(\cM(\Omega)\), and reduced limits \(\mu_\oplus^\#\) and \(\mu_\ominus^\#\), respectively. On the one hand, the monotonicity of the reduced limit \eqref{eq:monotonicityRL} implies that
\[
\mu_\ominus^\# \le \mu^\# \le \mu_\oplus^\#.
\]
On the other hand, we deduce from Fatou's lemma that
\[
\mu_\ominus \le \mu_\ominus^\# \quad \text{and} \quad \mu_\oplus^\# \le \mu_\oplus.
\]
Hence we have
\[
\mu_\ominus \le \mu^\# \le \mu_\oplus.
\]
This estimate implies that
\[
\norm{\mu^\#}_{\cM(\Omega)} \le \norm{\mu_\oplus}_{\cM(\Omega)} + \norm{\mu_\ominus}_{\cM(\Omega)}.
\]
It then follows from the lower semicontinuity of the total variation norm with respect to weak* convergence in \(\cM(\Omega)\) and identity~\eqref{eq:normDecomposition} that
\[
\norm{\mu^\#}_{\cM(\Omega)} \le \liminf_{n \to \infty} {} \norm{\mu_n^+}_{\cM(\Omega)} + \liminf_{n \to \infty} {} \norm{\mu_n^-}_{\cM(\Omega)} \le \liminf_{n \to \infty} {} \norm{\mu_n}_{\cM(\Omega)}.
\]
This concludes the proof of the proposition.
\end{proof}

% \cite{Ponce:2016}*{Proposition~2.6} Check that this is cited somewhere before

\section{Proof of \Cref{thm:existence}}
\label{sec:existence}

We first prove that \(F\) is lower semicontinuous with respect to the reduced limit: if the sequence of measures \((\mu_n)_{n \in \N}\) in \(\cM(\Omega)\) has a reduced limit \(\mu^\#\) and satisfies \(F(\mu_n) < \infty\) for each \(n \in \N\), then
\begin{equation}
\label{eq:lowerSemicontinuityF}
F(\mu^\#) \le \liminf_{n \to \infty} {} F(\mu_n).
\end{equation}
Taking a subsequence if necessary, we may assume that the limit inferior in the right-hand side of \eqref{eq:lowerSemicontinuityF} is an actual limit. Denote by \(u_n\) the unique solution of \eqref{eq:nonlinearDirichletProblem} corresponding to \(\mu_n\). By definition, the sequence \((u_n)_{n \in \N}\) converges strongly in \(L^1(\Omega)\) to the unique solution \(u^\#\) of \eqref{eq:nonlinearDirichletProblem} with datum \(\mu^\#\). Taking a further subsequence if needed, we can assume that \((u_n)_{n \in \N}\) converges almost everywhere to \(u^\#\) in \(\Omega\). On the one hand, we have \(u^\# - u_d \in L^p(\Omega)\) and
\begin{equation}
\label{eq:liminf}
\norm{u^\# - u_d}_{L^p(\Omega)} \le \liminf_{n \to \infty} {} \norm{u_n - u_d}_{L^p(\Omega)}.
\end{equation}
In the case where \(1 \le p < \infty\), this estimate is a consequence of Fatou's lemma. When \(p = \infty\), we have
\[
\abs{u^\# - u_d} = \lim_{n \to \infty} {} \abs{u_n - u_d} \le \liminf_{n \to \infty} {} \norm{u_n - u_d}_{L^\infty(\Omega)}
\]
almost everywhere in \(\Omega\). Hence \eqref{eq:liminf} holds for every \(1 \le p \le \infty\). On the other hand, \Cref{prop:lowerSemicontinuityRL} implies that
\begin{equation}
\label{eq:lowerSemicontinuityNorm}
\norm{\mu^\#}_{\cM(\Omega)} \le \liminf_{n \to \infty} {} \norm{\mu_n}_{\cM(\Omega)}.
\end{equation}
Combining \eqref{eq:liminf} and \eqref{eq:lowerSemicontinuityNorm} we obtain \eqref{eq:lowerSemicontinuityF}.

Now, let \((\mu_n)_{n \in \N}\) be a minimizing sequence of \(F\)---that is, a sequence in \(\cM(\Omega)\) such that
\[
\lim_{n \to \infty} {} F(\mu_n) = \inf_{\mu \in \cM(\Omega)} {} F(\mu).
\]
Since \(F \not\equiv \infty\), we may assume that each \(\mu_n\) is a good measure satisfying \(F(\mu_n) < \infty\). Taking a subsequence if necessary, we may also assume that \((\mu_n)_{n \in \N}\) has a reduced limit \(\mu^\#\). The lower semicontinuity of \(F\) with respect to the reduced limit then implies that
\[
F(\mu^\#) \le \liminf_{n \to \infty} {} F(\mu_n) = \inf_{\mu \in \cM(\Omega)} {} F(\mu).
\]
This proves that \(\mu^\#\) is a solution of \eqref{eq:optimalControlProblem}. The proof of the theorem is thus complete.

\section{Truncation with nonnegative supersolutions}
\label{sec:truncation}

Let \(\nu \in \cM(\Omega)\) and let \(v\) be the unique solution of the linear Dirichlet problem
\[
\left\{
\begin{alignedat}{3}
- \Delta v &= \nu &&\quad \text{in \(\Omega\),} \\
v &= 0 &&\quad \text{on \(\partial\Omega\).}
\end{alignedat}
\right.
\]
Then, for every \(\kappa \ge 0\), the distribution \(\Delta \min {} \set{u, \kappa}\) is a finite measure on \(\Omega\) and
\[
\norm{\Delta \min {} \set{u, \kappa}}_{\cM(\Omega)} \le \norm{\Delta u}_{\cM(\Omega)};
\]
see~\cite{BrezisPonce:2008}*{Theorem~1.2}. The goal of this section is two-fold: to extend this result to the case where \(\kappa\) is not a constant and to find a counterpart of the estimate for the operator \(- \Delta + g\) with an absorption term.

Let us define
\[
X(\Omega) \defeq \set*{v \in W^{1, 1}(\Omega) : \Delta v \in \cM(\Omega)}
\]
and
\[
X_0(\Omega) \defeq X(\Omega) \cap W_0^{1, 1}(\Omega).
\]
The main result of this section is the following:

\begin{proposition}
\label{prop:truncationSupersolution}
Let \(\mu \in \cM(\Omega)\) be a good measure and let \(u\) be the unique solution of \eqref{eq:nonlinearDirichletProblem} with datum \(\mu\). Then, for every nonnegative function \(w \in X(\Omega)\) such that \(g(w) \in L^1(\Omega)\) and
\[
- \Delta w + g(w) \ge 0 \quad \text{in the sense of distributions in \(\Omega\),}
\]
the function \(z \defeq \min {} \set{u, w}\) satisfies \(z \in X_0(\Omega)\), \(g(z) \in L^1(\Omega)\) and
\[
\norm{- \Delta z + g(z)}_{\cM(\Omega)} \le \norm{- \Delta u + g(u)}_{\cM(\Omega)}.
\]
\end{proposition}

We recall that the equation
\[
- \Delta w + g(w) \ge 0
\]
holds in the sense of distributions in \(\Omega\) if
\[
- \int_\Omega {} w \Delta \varphi + \int_\Omega {} g(w) \varphi  \ge 0,
\]
for every nonnegative function \(\varphi \in C_c^\infty(\Omega)\).

\Cref{prop:truncationSupersolution} is a straightforward consequence of

\begin{lemma}
\label{lem:truncation}
For \(i \in \set{1, 2}\), let \(u_i \in X(\Omega)\) and \(a_i \in L^1(\Omega)\). If \(u_1 \le u_2\) on \(\partial\Omega\) and \(- \Delta u_2 + a_2 \ge 0\) in \(\Omega\), then the functions
\[
u \defeq \min {} \set{u_1, u_2} \quad \text{and} \quad a \defeq \begin{cases}
a_1 & \text{in \(\set{u_1 \le u_2}\),} \\
a_2 & \text{in \(\set{u_2 < u_1}\),}
\end{cases}
\]
satisfy \(u \in X(\Omega)\), \(a \in L^1(\Omega)\) and
\[
\int_\Omega {} \abs{- \Delta u + a} \le \int_\Omega {} \abs{- \Delta u_1 + a_1} + \int_{\set{u_1 > u_2}} {} (a_2 - a_1).
\]
\end{lemma}

\begin{proof}[Proof of \Cref{lem:truncation}]

We divide the proof of \Cref{lem:truncation} into three steps.

\begin{step}
The conclusion holds if \(u_i, a_i \in C^\infty(\overline{\Omega})\) and the set \(\set{u_1 = u_2}\) is a smooth compact manifold without boundary.
\end{step}

The assumption \(u_1 \le u_2\) on \(\partial\Omega\) implies that \(\partial \set{u_1 > u_2} = \set{u_1 = u_2}\). By the Divergence theorem, for every \(\varphi \in C_c^\infty(\Omega)\), we have
\[
\int_{\set{u_1 > u_2}} {} (u_1 - u_2) \Delta \varphi = \int_{\set{u_1 > u_2}} {} \varphi \Delta (u_1 - u_2) - \int_{\set{u_1 = u_2}} {} \varphi \frac{\partial}{\partial n} (u_1 - u_2) \d\sigma,
\]
where \(n\) is the unit normal vector on \(\set{u_1 = u_2}\) pointing outwards with respect to \(\set{u_1 > u_2}\) and \(\sigma\) denotes the surface measure of \(\set{u_1 = u_2}\) that coincides with the Hausdorff measure \(\cH^{N - 1} \lfloor_{\set{u_1 = u_2}}\). The equation
\[
\Delta (u_1 - u_2)^+ = \chi_{\set{u_1 > u_2}} \Delta (u_1 - u_2) - \frac{\partial}{\partial n} (u_1 - u_2) \cH^{N - 1} \lfloor_{\set{u_1 = u_2}}
\]
thus holds in the sense of distributions in \(\Omega\). Since
\[
u = u_1 - (u_1 - u_2)^+
\]
and the set \(\set{u_1 = u_2}\) is negligible for the Lebesgue measure, one has
\begin{multline*}
- \Delta u + a = \chi_{\set{u_1 < u_2}} (- \Delta u_1 + a_1) + \chi_{\set{u_1 > u_2}} (- \Delta u_2 + a_2) \\ - \frac{\partial}{\partial n} (u_1 - u_2) \cH^{N-1} \lfloor_{\set{u_1 = u_2}}
\end{multline*}
in the sense of distributions in \(\Omega\); thus in the sense of measures on \(\Omega\). Notice that the minimum of the function \(u_1 - u_2\) in \(\set{u_1 \ge u_2}\) is achieved on \(\set{u_1 = u_2}\); whence we have \(\frac{\partial}{\partial n} (u_1 - u_2) \le 0\) on \(\set{u_1 = u_2}\). Computing the total variation of \(- \Delta u + a\) yields
\begin{multline}
\label{eq:normBigEq}
\int_\Omega {} \abs{- \Delta u + a} \le \int_{\set{u_1 < u_2}} {} \abs{- \Delta u_1 + a_1} + \int_{\set{u_1 > u_2}} {} \abs{- \Delta u_2 + a_2} \\ - \int_{\set{u_1 = u_2}} {} \frac{\partial}{\partial n} (u_1 - u_2) \d\sigma.
\end{multline}
On the other hand, by the Divergence theorem, we have
\begin{multline}
\label{eq:normalDerivativeBigEq}
- \int_{\set{u_1 = u_2}} {} \frac{\partial}{\partial n} (u_1 - u_2) \d\sigma = - \int_{\set{u_1 > u_2}} {} \Delta (u_1 - u_2) \\
= \int_{\set{u_1 > u_2}} {} (- \Delta u_1 + a_1) - \int_{\set{u_1 > u_2}} {} (- \Delta u_2 + a_2) + \int_{\set{u_1 > u_2}} {} (a_2 - a_1).
\end{multline}
Since \(- \Delta u_2 + a_2 \ge 0\), we also have
\begin{equation}
\label{eq:normNonnegativeMeasure}
\int_{\set{u_1 > u_2}} {} (- \Delta u_2 + a_2) = \int_{\set{u_1 > u_2}} {} \abs{- \Delta u_2 + a_2}.
\end{equation}
Combining \eqref{eq:normBigEq}, \eqref{eq:normalDerivativeBigEq} and \eqref{eq:normNonnegativeMeasure} we obtain
\[
\int_\Omega {} \abs{- \Delta u + a} \le \int_\Omega {} \abs{- \Delta u_1 + a_1} + \int_{\set{u_1 > u_2}} {} (a_2 - a_1).
\]

\begin{step}
The conclusion holds if \(u_i\) and \(a_i\) are as in the statement and the set \(\set{u_1 = u_2}\) has Lebesgue measure zero.
\end{step}

Let \(\omega \Subset \Omega\) be an open subset and let \((\rho_n)_{n \in \N_*}\) be a sequence of mollifiers in \(C_c^\infty(\R^N)\) such that \(\omega - \supp \rho_n \Subset \Omega\). Fubini's theorem implies that
\begin{equation}
\label{eq:regularized}
- \Delta (\rho_n * u_i) + \rho_n * a_i = \rho_n * \mu_i \quad \text{in the sense of distributions in \(\omega\);}
\end{equation}
see e.g.~the proof of Proposition~2.7 in \cite{Ponce:2016}. We seek to apply Step~1 to the regularized functions \(\rho_n * u_i\) and \(\rho_n * a_i\). For this purpose, as \(0\) need not be a regular value of \(\rho_n * (u_1 - u_2)\), we use the Morse--Sard theorem~\cite{Willem:2013}*{Theorem~7.4.3} to deduce that for each \(n \in \N_*\) there exists a regular value \(t_n\) of \(\rho_n * (u_1 - u_2)\) such that \(0 \le t_n \le 1/n\). Let us then define the functions
\[
z_n \defeq \min {} \set{\rho_n * u_1, \rho_n * u_2 + t_n}
\]
and
\[
b_n \defeq \begin{cases}
\rho_n * a_1 & \text{in \(\set{\rho_n * u_1 \le \rho_n * u_2 + t_n}\),} \\
\rho_n * a_2 & \text{in \(\set{\rho_n * u_2 + t_n < \rho_n * u_1}\).}
\end{cases}
\]
Applying Step~1 to the functions \(\rho_n * u_1\) and \(\rho_n * u_2 + t_n\) in \(\omega\) we obtain
\begin{multline}
\label{eq:regularizedBigEstimate}
\int_\omega {} \abs{- \Delta z_n + b_n} \le \int_\omega {} \abs{- \Delta (\rho_n * u_1) + \rho_n * a_1} \\ + \int_{\set{\rho_n * u_1 > \rho_n * u_2 + t_n} \cap \omega} {} \rho_n * (a_2 - a_1).
\end{multline}
Since the sequence \((\rho_n * (a_2 - a_1))_{n \in \N_*}\) converges strongly to \(a_2 - a_1\) in \(L^1(\Omega)\), by the partial converse of the Dominated convergence theorem, passing to a subsequence if necessary, there exists \(f \in L^1(\Omega)\) such that
\[
\abs{\rho_n * (a_2 - a_1)} \le f \quad \text{almost everywhere in \(\Omega\);}
\]
see e.g.~\cite{Willem:2013}*{Proposition~4.2.10}. Taking a further subsequence if needed, we may assume that the sequences \((\rho_n * (a_2 - a_1))_{n \in \N_*}\) and \((\rho_n * u_i)_{n \in \N_*}\) converge almost everywhere in \(\Omega\) to the functions \(a_2 - a_1\) and \(u_i\), respectively. Since \(\set{u_1 = u_2}\) has Lebesgue measure zero, it thus follows that the characteristic functions \(\chi_{\set{\rho_n * u_1 > \rho_n * u_2 + t_n}}\) converge almost everywhere to \(\chi_{\set{u_1 > u_2}}\) in \(\Omega\) as \(n\) tends to infinity. By the Dominated convergence theorem, we then have
\[
\lim_{n \to \infty} {} \int_{\set{\rho_n * u_1 > \rho_n * u_2 + t_n} \cap \omega} \rho_n * (a_2 - a_1) = \int_{\set{u_1 > u_2} \cap \omega} {} (a_2 - a_1).
\]
On the other hand, it follows from \eqref{eq:regularized} that
\[
\lim_{n \to \infty} {} \int_\omega {} \abs{- \Delta (\rho_n * u_1) + \rho_n * a_1} = \norm{- \Delta u_1 + a_1}_{\cM(\omega)};
\]
see also~\cite{Ponce:2016}*{Proposition~2.7}. Estimate \eqref{eq:regularizedBigEstimate} then implies that
\[
\int_\omega {} \abs{- \Delta u + a} \le \norm{- \Delta u_1 + a_1}_{\cM(\omega)} + \int_{\set{u_1 > u_2} \cap \omega} (a_2 - a_1).
\]
Let us now take a nondecreasing sequence \((\omega_k)_{k \in \N}\) of open subsets of \(\Omega\) such that
\[
\bigcup_{k = 0}^\infty \omega_k = \Omega.
\]
Applying the previous estimate with \(\omega \defeq \omega_k\) and letting \(k\) tend to infinity, the conclusion follows from the Monotone set lemma.

\begin{step}
Proof of \Cref{lem:truncation} completed.
\end{step}

Let \(S \defeq \set{s \in \R : \meas{\set{u_1 - u_2 = s}} \neq 0}\). Since the set \(S\) is countable, its complement is dense in \(\R\). Hence, there exists a sequence \((s_k)_{k \in \N}\) of nonnegative numbers in \(\R \setminus S\) converging to \(0\). Applying Step~2 with the function \(u_2 + s_k\) instead of \(u_2\) and taking the limit as \(k\) tends to infinity in the given estimate, the conclusion follows from the Dominated convergence theorem. The proof of the lemma is complete.
\end{proof}

We now proceed with the

\begin{proof}[Proof of \Cref{prop:truncationSupersolution}]
Let \(w \in X(\Omega)\) be a nonnegative function such that \(g(w) \in L^1(\Omega)\) and
\[
- \Delta w + g(w) \ge 0 \quad \text{in the sense of distributions in \(\Omega\).}
\]
The function \(z \defeq \min {} \set{u, w}\) belongs to \(W_0^{1, 1}(\Omega)\). We then deduce from the preceding lemma that \(z \in X_0(\Omega)\), \(g(z) \in L^1(\Omega)\) and
\[
\int_\Omega {} \abs{- \Delta z + g(z)} \le \int_\Omega {} \abs{- \Delta u + g(u)} + \int_{\set{u > w}} {} (g(w) - g(u)).
\]
Since \(g\) is a nondecreasing function, the last integral is nonpositive; whence
\[
\int_\Omega {} \abs{- \Delta z + g(z)} \le \int_\Omega {} \abs{- \Delta u + g(u)},
\]
and the proof is complete.
\end{proof}

\begin{remark}
\label{rem:truncationSubsolution}
\Cref{prop:truncationSupersolution} has an immediate counterpart for truncation with nonpositive subsolutions of the operator \(- \Delta + g\). More precisely, if \(w \in X(\Omega)\) is a nonpositive function such that \(g(w) \in L^1(\Omega)\) and
\[
- \Delta w + g(w) \le 0 \quad \text{in the sense of distributions in \(\Omega\),}
\]
then the function \(z \defeq \max {} \set{u, w}\) satisfies \(z \in X_0(\Omega)\), \(g(z) \in L^1(\Omega)\) and
\[
\norm{- \Delta z + g(z)}_{\cM(\Omega)} \le \norm{- \Delta u + g(u)}_{\cM(\Omega)}.
\]
This property follows from \Cref{prop:truncationSupersolution} applied to \(-u\), \(-w\) and \(\tilde{g}(t) = - g(-t)\).
\end{remark}

\section{Proofs of \Cref{thm:nonnegative,thm:bounded}}
\label{sec:regularity}

\Cref{thm:nonnegative,thm:bounded} can be deduced from the following statement:

\begin{proposition}
\label{prop:boundedSupersolution}
Let \(w \in X(\Omega)\) be a nonnegative function such that \(g(w) \in L^1(\Omega)\) and
\[
- \Delta w + g(w) \ge 0 \quad \text{in the sense of distributions in \(\Omega\).}
\]
If \(u_d \le w\) almost everywhere in \(\Omega\), then \(u^* \le w\) almost everywhere in \(\Omega\).
\end{proposition}

%The assumption \(F \not\equiv \infty\) guarantees that \(u^* - u_d \in L^p(\Omega)\). This is the case for example when \(u_d \in L^\infty(\Omega)\).

Indeed, if \(a\) is a nonnegative number such that \(u_d \le a\), then we have \(u^* \le a\). Similarly, if \(b\) is a nonpositive number such that \(u_d \ge b\), then \(u^* \ge b\). Hence, when the desired state \(u_d\) is nonnegative, we have that \(u^*\) is also nonnegative, which is precisely the statement of \Cref{thm:nonnegative}. One also deduces \Cref{thm:bounded} when \(u_d\) is bounded by taking \(a = \norm{u_d}_{L^\infty(\Omega)}\) and \(b = - \norm{u_d}_{L^\infty(\Omega)}\). We are thus left with the

\begin{proof}[Proof of \Cref{prop:boundedSupersolution}]
By \Cref{prop:truncationSupersolution}, the function \(v \defeq \min {} \set{u^*, w}\) satisfies \(v \in X_0(\Omega)\), \(g(v) \in L^1(\Omega)\) and
\[
\norm{- \Delta v + g(v)}_{\cM(\Omega)} \le \norm{- \Delta u^* + g(u^*)}_{\cM(\Omega)}.
\]
Observe that \(v\) solves \eqref{eq:nonlinearDirichletProblem} with datum \(\mu = - \Delta v + g(v)\); whence
\[
F(- \Delta v + g(v)) = \norm{v - u_d}_{L^p(\Omega)} + \alpha \norm{- \Delta v + g(v)}_{\cM(\Omega)}.
\]
Since \(u_d \le w\) almost everywhere in \(\Omega\), one has
\[
\norm{v - u_d}_{L^p(\Omega)} \le \norm{u^* - u_d}_{L^p(\Omega)}.
\]
Hence
\[
F(- \Delta v + g(v)) \le F(\mu^*).
\]
Since \(\mu^*\) minimizes \(F\), equality must hold and one then deduces that
\[
\norm{v - u_d}_{L^p(\Omega)} = \norm{u^* - u_d}_{L^p(\Omega)}.
\]
By assumption, we have \(\abs{w - u_d} < \abs{u^* - u_d}\) almost everywhere in \(\set{u^* > w}\). It then follows from the equality above that \(\set{u^* > w}\) must have Lebesgue measure zero. Thus, \(u^* \le w\) almost everywhere in \(\Omega\), and the proof is complete.
\end{proof}

\section{Regularization of the ideal state and stability analysis}
\label{sec:regularization}

Throughout this section, we assume for convenience that \(u_d \in L^p(\Omega)\). We begin by studying the convergence of optimal states \(u^*\) as the control parameter \(\alpha\) converges to \(0\). As we shall see, such a consideration allows one to approximate \(u_d\) by solutions of \eqref{eq:nonlinearDirichletProblem} in the \(L^p\)-scale. The use of the cost functional \(F\) can thus be seen as a penalization strategy aimed at obtaining a new function \(u^*\) close to \(u_d\) with, hopefully, better properties.

\begin{proposition}
\label{prop:regularizationIdealState}
Assume that \(1 \le p < \infty\) and \(u_d \in L^p(\Omega)\). Given a sequence of positive numbers \((\alpha_n)_{n \in \N}\) converging to \(0\), let \(F_{\alpha_n}\) be the cost functional with control parameter \(\alpha_n\), let \(\mu_n^*\) be a minimizer of \(F_{\alpha_n}\) and let \(u_n^*\) be the optimal state associated to \(\mu_n^*\). Then, we have
\begin{enumerate}[(i)]
\item \((u_n^*)_{n \in \N}\) converges to \(u_d\) in \(L^p(\Omega)\);
\item \((\mu_n^*)_{n \in \N}\) is bounded in \(\cM(\Omega)\) if and only if \(u_d\) solves \eqref{eq:nonlinearDirichletProblem} with some datum \(\mu^\#\).
\end{enumerate}
\end{proposition}

\begin{proof}
Since \(p < \infty\) and \(u_d \in L^p(\Omega)\), we have
\begin{equation}
\label{eq:limInfFnZero}
\lim_{n \to \infty} \inf_{\mu \in \cM(\Omega)} {} F_{\alpha_n}(\mu) = 0.
\end{equation}
Indeed, as \(\mu_n^*\) minimizes \(F_{\alpha_n}\), for every \(\varphi \in C_c^\infty(\Omega)\), we have
\[
F_{\alpha_n}(\mu_n^*) \le F_{\alpha_n}(- \Delta \varphi + g(\varphi)) = \norm{\varphi - u_d}_{L^p(\Omega)} + \alpha_n \norm{- \Delta \varphi + g(\varphi)}_{\cM(\Omega)}.
\]
Letting \(n\) tend to infinity, we obtain
\[
\limsup_{n \to \infty} {} F_{\alpha_n}(\mu_n^*) \le \norm{\varphi - u_d}_{L^p(\Omega)}.
\]
Since \(p < \infty\), by density of \(C_c^\infty(\Omega)\) in \(L^p(\Omega)\), the left-hand side must vanish. Hence \eqref{eq:limInfFnZero} holds. In particular, the sequence \((u_n)_{n \in \N}\) satisfies Property~\emph{(i)}.

For the second part of the proposition, first assume that the sequence \((\mu_n^*)_{n \in \N}\) is bounded in \(\cM(\Omega)\). Taking a subsequence if necessary, we may assume that \((\mu_n^*)_{n \in \N}\) has a reduced limit \(\mu^\#\). By definition of the reduced limit and Property~\emph{(i)},
\[
\norm{u^\# - u_d}_{L^p(\Omega)} = 0;
\]
whence \(u_d\) satisfies \eqref{eq:nonlinearDirichletProblem} with datum \(\mu^\#\). This proves the direct implication in Property~\emph{(ii)}. The reverse implication readily follows from
\[
\norm{\mu_n^*}_{\cM(\Omega)} \le \frac{F_{\alpha_n}(\mu_n^*)}{\alpha_n} \le \frac{F_{\alpha_n}(- \Delta u_d + g(u_d))}{\alpha_n} = \norm{- \Delta u_d + g(u_d)}_{\cM(\Omega)}.
\]
This concludes the proof of the proposition.
\end{proof}

\begin{remark}
\Cref{prop:regularizationIdealState} is false when \(p = \infty\) for an arbitrary \(u_d \in L^\infty(\Omega)\). Indeed, observe that for each \(n \in \N\) the precise representative \(\widehat{u}_n\) of \(u_n\) is a quasi-continuous function with respect to the Newtonian or \(\DL{1}\) capacity (see~p.~\pageref{eq:capacity}). If \((u_n)_{n \in \N}\) converges to \(u_d\) in \(L^\infty(\Omega)\), then \(u_d\) must be equal almost everywhere to the quasi-continuous function \(\lim\limits_{n \to \infty} {} \widehat{u}_n\). However, such a property does not hold for an arbitrary function in \(L^\infty(\Omega)\).
\end{remark}

We now focus on the stability of \eqref{eq:optimalControlProblem} with respect to the ideal state \(u_d\). More precisely, we prove

\begin{proposition}
\label{prop:stabilityIdealState}
Assume that \(1 \le p \le \infty\) and \(u_d \in L^p(\Omega)\). Given a sequence \((u_{d, n})_{n \in \N}\) of functions in \(L^p(\Omega)\) converging strongly to \(u_d\) in \(L^p(\Omega)\), let \(F_{u_{d, n}}\) be the cost functional with ideal state \(u_{d, n}\), let \(\mu_n^*\) be a minimizer of \(F_{u_{d, n}}\) and let \(u_n^*\) be the optimal state associated to \(\mu_n^*\). If the sequence \((\mu_n^*)_{n \in \N}\) has a reduced limit \(\mu^\#\), then \(\mu^\#\) is a solution of \eqref{eq:optimalControlProblem}.
\end{proposition}

\begin{proof}
The triangle inequality implies that, for every \(n \in \N\),
\[
F_{u_d}(\mu_n^*) \le F_{u_{d, n}}(\mu_n^*) + \norm{u_{d, n} - u_d}_{L^p(\Omega)}.
\]
Let \(\mu^*\) be a minimizer of \(F_{u_d}\) with associated state \(u^*\). Since \(\mu_n^*\) minimizes \(F_{u_{d, n}}\), we deduce from the previous inequality that
\[
F_{u_d}(\mu_n^*) \le F_{u_{d, n}}(\mu^*) + \norm{u_{d, n} - u_d}_{L^p(\Omega)}.
\]
Since \((\mu_n^*)_{n \in \N}\) has reduced limit \(\mu^\#\), the sequence \((u_n^*)_{n \in \N}\) converges strongly in \(L^1(\Omega)\) to some function \(u^\#\) which is the unique solution of \eqref{eq:nonlinearDirichletProblem} corresponding to \(\mu^\#\). It then follows from the lower semicontinuity of \(F_{u_d}\) with respect to the reduced limit (see the proof of \Cref{thm:existence}) that
\[
F_{u_d}(\mu^\#) \le \liminf_{n \to \infty} {} F_{u_d}(\mu_n^*) \le F_{u_d}(\mu^*).
\]
This implies that \(\mu^\#\) is a minimizer of \(F_{u_d}\).
\end{proof}

In \Cref{prop:stabilityIdealState}, the requirement that \((\mu_n^*)_{n \in \N}\) has a reduced limit is not restrictive since this is true up to the extraction of a subsequence. Indeed, since \(\mu_n^*\) minimizes \(F_{u_{d, n}}\),
\[
\norm{\mu_n^*}_{\cM(\Omega)} \le \frac{\norm{u_{d, n}}_{L^p(\Omega)}}{\alpha}.
\]
The assumption that \((u_{d, n})_{n \in \N}\) converges strongly in \(L^p(\Omega)\) thus implies that \((\mu_n^*)_{n \in \N}\) is bounded in \(\cM(\Omega)\) and one can extract a subsequence having a reduced limit.

\section{Study of the concentrated part of \(\mu^*\)}
\label{sec:idealStateMeasure}

In this section, we assume that \(u_d \in L^1(\Omega)\) and \(\Delta u_d \in \cM(\Omega)\). The latter property means that there exists a finite measure \(\nu\) on \(\Omega\) such that
\[
\Delta u_d = \nu \quad \text{in the sense of distributions in \(\Omega\).}
\]
We then identify \(\Delta u_d\) with \(\nu\). Our main goal is to show that every \(\mu \in \cM(\Omega)\) such that \(F(\mu) < \infty\) agrees with \(- \Delta u_d\) on every sufficiently small subset of \(\Omega\). The notion of smallness is measured in terms of the following capacity depending on \(p > 1\): for every compact set \(K \subset \Omega\), the \(\DL{p'}\) capacity of \(K\) relative to \(\Omega\) is defined by
\[
\label{eq:capacity}
\capt_{\DL{p'}}(K; \Omega) = \inf {} \set[\big]{\norm{\Delta \zeta}_{L^{p'}(\Omega)}^{p'} : \text{\(\zeta \in C_0^\infty(\overline{\Omega})\) is nonnegative and \(\zeta > 1\) in \(K\)}},
\]
where \(p'\) is the conjugate exponent of \(p\) and \(C_0^\infty(\overline{\Omega})\) denotes the space of functions \(\zeta \in C^\infty(\overline{\Omega})\) such that \(\zeta = 0\) on \(\partial\Omega\).

We then extend the \(\DL{p'}\) capacity to Borel sets using a standard regularity procedure: the capacity of an open set \(U \subset \Omega\) is defined as the supremum of the capacity of compact sets \(K \subset \Omega\) contained in \(U\), and the capacity of a Borel set \(A \subset \Omega\) is defined as the infimum of the capacity of open sets \(U \subset \Omega\) containing \(A\).

When \(1 < p < \infty\), one deduces from the Cald\'eron--Zygmund \(L^{p'}\) estimates~\cite{GilbargTrudinger:2001}*{Corollary~9.10} that the \(\DL{p'}\) capacity vanishes on the same sets as the Sobolev (or Bessel) \(W^{2, p'}\) capacity. More precisely, for every Borel set \(A \subset \Omega\), one has
\begin{equation}
\label{eq:LaplacianSobolevCapacities}
\capt_{\DL{p'}}(A; \Omega) = 0 \quad \text{if and only if} \quad \capt_{W^{2, p'}}(A) = 0.
\end{equation}
In the case where \(p = \infty\), the same property holds with respect to the \(W^{1, 2}\) (or Newtonian) capacity; see~\cite{BrezisMarcusPonce:2007}*{Theorem~4.E.1}.

The space \(C_0^\infty(\overline{\Omega})\) provides an alternative formulation of the concept of solution of \eqref{eq:nonlinearDirichletProblem}. Indeed, one shows that \(u\) is a solution of \eqref{eq:nonlinearDirichletProblem} with datum \(\mu \in \cM(\Omega)\) if and only if \(u \in L^1(\Omega)\), \(g(u) \in L^1(\Omega)\) and, for every \(\zeta \in C_0^\infty(\overline{\Omega})\),
\[
- \int_\Omega {} u \Delta \zeta + \int_\Omega {} g(u) \zeta = \int_\Omega {} \zeta \d\mu \, ;
\]
see e.g.~\cite{Ponce:2016}*{Proposition~6.3}. This latter formulation of \eqref{eq:nonlinearDirichletProblem} has been introduced by Littman, Stampacchia and Weinberger~\cite{LittmanStampacchiaWeinberger:1963}*{Definition~5.1} and implicitly encodes the zero boundary condition.

In the case where the ideal state \(u_d\) belongs to \(L^p(\Omega)\), every solution of \eqref{eq:nonlinearDirichletProblem} with datum \(\mu \in \cM(\Omega)\) such that \(F(\mu) < \infty\) also belongs to \(L^p(\Omega)\). This implies that the measure \(\mu\) is diffuse with respect to the \(\DL{p'}\) capacity, where by \emph{diffuse} we mean that for every Borel set \(A \subset \Omega\) such that \(\capt_{\DL{p'}}(A; \Omega) = 0\), we have \(\mu(A) = 0\). When \(u_d \not\in L^p(\Omega)\), a concentrated part might appear. Such a phenomenon can be quantified by the main result of this section:

\begin{proposition}
\label{prop:concentratedPart}
Assume that \(\frac{N}{N - 2} \le p \le \infty\) and \(\Delta u_d \in \cM(\Omega)\). If \(\mu \in \cM(\Omega)\) is such that \(F(\mu) < \infty\), then
\[
\mu_\c = (- \Delta u_d)_\c,
\]
where the subscript \(\c\) denotes the concentrated part of the measure with respect to the \(\DL{p'}\) capacity. In particular, \(\mu_\c^*\) is uniquely determined in terms of \(\Delta u_d\).
\end{proposition}

We recall that a measure \(\nu \in \cM(\Omega)\) is \emph{concentrated} with respect to the \(\DL{p'}\) capacity if there exists a Borel set \(F \subset \Omega\) such that \(\capt_{\DL{p'}}(F; \Omega) = 0\) and
\[
\abs{\nu}(\Omega \setminus F) = 0.
\]
By a counterpart of the Lebesgue decomposition theorem involving capacities, each finite Borel measure has a unique decomposition as a sum of diffuse and concentrated measures; see~\citelist{\cite{Mokobodzki:1978} \cite{Ponce:2016}*{Proposition~14.12}}.

\begin{proof}[Proof of \Cref{prop:concentratedPart}]
Let \(u\) be the unique solution of \eqref{eq:nonlinearDirichletProblem} with datum \(\mu\). Since \(F(\mu) < \infty\), the function \(u - u_d\) belongs to \(L^p(\Omega)\). We then deduce from the H\"older inequality that, for every \(\varphi \in C_c^\infty(\Omega)\), the following estimate holds:
\begin{equation}
\label{eq:LpIneqMeasure}
\abs[\bigg]{\int_\Omega {} \varphi \Delta (u - u_d)} = \abs[\bigg]{\int_\Omega {} (u - u_d) \Delta \varphi} \le \norm{u - u_d}_{L^p(\Omega)} \norm{\Delta \varphi}_{L^{p'}(\Omega)}.
\end{equation}
We show that the finite measure \(\Delta (u - u_d)\) is diffuse with respect to the \(\DL{p'}\) capacity. Indeed, let \(K \subset \Omega\) be a compact set such that \(\capt_{\DL{p'}}(K; \Omega) = 0\). Take a sequence \((\varphi_n)_{n \in \N}\) in \(C_c^\infty(\Omega)\) such that
\begin{enumerate}[(a)]
\item \((\varphi_n)_{n \in \N}\) converges pointwise to the characteristic function \(\chi_K\);
\item \((\varphi_n)_{n \in \N}\) is bounded in \(L^\infty(\Omega)\);
\item \((\Delta \varphi_n)_{n \in \N}\) converges to \(0\) in \(L^{p'}(\Omega)\);
\end{enumerate}
see~\cite{PonceWilmet:2017}*{Proposition~3.1}. Applying estimate \eqref{eq:LpIneqMeasure} to the sequence \((\varphi_n)_{n \in \N}\) and letting \(n\) tend to infinity, we deduce from the Dominated convergence theorem that
\[
\Delta (u - u_d) (K) = \int_K \Delta (u - u_d) = 0.
\]
Hence
\[
\Delta u (K) = \Delta u_d (K).
\]
Since \(K\) has Lebesgue measure zero, one also has
\[
\int_K g(u) = 0.
\]
On the other hand, we have
\[
- \Delta u + g(u) = \mu \quad \text{in the sense of measures on \(\Omega\).}
\]
Hence
\[
\mu(K) = - \Delta u (K) = - \Delta u_d (K).
\]
We have thus proved that
\[
\mu_\c = (- \Delta u_d)_\c
\]
on every compact subset of \(\Omega\). This equality also holds on every Borel subset of \(\Omega\) by inner regularity. The proof of the proposition is complete.
\end{proof}

We also have the following property without restriction on the exponent \(p\):

\begin{proposition}
\label{prop:concentratedPartNewtonian}
Let \(v \in X(\Omega)\). If {} \(\abs{u_d} \le v\) almost everywhere in \(\Omega\), then
\[
\abs{\mu^*}_\c \le \abs{- \Delta v}_\c,
\]
where the subscript \(\c\) denotes the concentrated part of the measure with respect to the \(\DL{1}\) capacity.
\end{proposition}

The proof of \Cref{prop:concentratedPartNewtonian} relies on several ingredients. The first one is the following weak maximum principle: if \(u \in L^1(\Omega)\) is such that
\[
- \int_\Omega {} u \Delta \zeta \ge 0,
\]
for every nonnegative function \(\zeta \in C_0^\infty(\overline{\Omega})\), then \(u \ge 0\) almost everywhere in \(\Omega\); see~\cite{Ponce:2016}*{Proposition~6.1}. The second ingredient is an analogue of \Cref{prop:boundedSupersolution} for nonnegative superharmonic functions. More precisely, if \(w \in W^{1, 1}(\Omega)\) is a nonnegative function such that \(\Delta w \in \cM(\Omega)\) and
\[
- \Delta w \ge 0 \quad \text{in the sense of distributions in \(\Omega\),}
\]
and if \(u_d \le w\) almost everywhere in \(\Omega\), then \(u^* \le w\) almost everywhere in \(\Omega\). This result also has a natural counterpart for nonpositive subharmonic functions (cf.~\Cref{rem:truncationSubsolution}). The third and last ingredient is the inverse maximum principle~\cite{DupaignePonce:2004}*{Theorem~3}: if \(u \in L^1(\Omega)\) is such that \(\Delta u \in \cM(\Omega)\) and \(u \ge 0\) almost everywhere in \(\Omega\), then the concentrated part of \(- \Delta u\) with respect to the \(\DL{1}\) capacity satisfies
\[
(- \Delta u)_\c \ge 0;
\]
see also~\cite{Ponce:2016}*{Proposition~6.13}.

\begin{proof}[Proof of \Cref{prop:concentratedPartNewtonian}]
The unique solution \(w\) of the linear Dirichlet problem
\[
\left\{
\begin{alignedat}{3}
- \Delta w &= (- \Delta v)^+ &&\quad \text{in \(\Omega\),} \\
w &= 0 &&\quad \text{on \(\partial\Omega\).}
\end{alignedat}
\right.
\]
satisfies
\[
- \int_\Omega {} (w - v) \Delta \zeta \ge 0,
\]
for every nonnegative function \(\zeta \in C_0^\infty(\overline{\Omega})\). It thus follows from the weak maximum principle that \(v \le w\) almost everywhere in \(\Omega\). Since \(u_d \le v\) almost everywhere in \(\Omega\), we deduce that \(u^* \le w\) almost everywhere in \(\Omega\) (cf.~\Cref{prop:boundedSupersolution}). The inverse maximum principle then implies that
\[
(- \Delta u^*)_\c \le (- \Delta w)_\c;
\]
whence
\[
\mu_\c^* \le (- \Delta w)_\c = (- \Delta v)_\c^+ \le \abs{- \Delta v}_\c.
\]
The same argument as above applied to the unique solution of the linear Dirichlet problem
\[
\left\{
\begin{alignedat}{3}
- \Delta w &= - (- \Delta v)^- &&\quad \text{in \(\Omega\),} \\
w &= 0 &&\quad \text{on \(\partial\Omega\),}
\end{alignedat}
\right.
\]
yields the estimate
\[
- \mu_\c^* \le (- \Delta v)^- \le \abs{- \Delta v}_\c.
\]
Combining the last two inequalities, we obtain the conclusion.
\end{proof}

As a direct consequence of \Cref{prop:concentratedPartNewtonian}, we see that if \(u_d \in L^\infty(\Omega)\), then \(\mu^*\) is diffuse with respect to the \(\DL{1}\) capacity; see~\cite{CasasKunisch:2014}*{Theorem~5.1} in the subcritical case. The same conclusion holds when \(u_d \in X(\Omega)\) and \(\nabla u_d \in L^2(\Omega)\), since in this case we have \(\Delta u_d \in (W_0^{1, 2}(\Omega))'\) and it is known that measures in this dual space are diffuse; see~\cite{Grun-Rehomme:1977}*{Proposition~1}.

\section{Optimal controls which are not summable functions}
\label{sec:nonexistenceSummableControl}

Our goal in this section is to show that solutions of \eqref{eq:optimalControlProblem} need not be in \(L^1(\Omega)\). For this purpose, we rely on \Cref{prop:concentratedPart} above and we assume that \(\frac{N}{N - 2} \le p \le \infty\). We begin by justifying the example given in the introduction, namely

\begin{proposition}
\label{prop:nonexistenceSummableControl}
Assume that \(N \ge 3\), \(\frac{N}{N - 2} \le p < \infty\) and \(g\) satisfies
\[
\abs{g(t)} \le C (\abs{t}^q + 1)
\]
for some constant \(C > 0\) and \(1 \leq q < p\). Then there exists \(u_d \in L^q(\Omega)\) such that the cost functional \(F\) with desired state \(u_d\) satisfies
\[
\text{\(F \not\equiv \infty\) \ in \(\cM(\Omega)\)} \quad \text{and} \quad \text{\(F \equiv \infty\) \ in \(L^1(\Omega)\).}
\]
In particular, the \emph{Lavrentiev phenomenon} occurs:
\[
\inf_{\mu \in \cM(\Omega)} {} F(\mu) < \inf_{\mu \in L^1(\Omega)} {} F(\mu).
\]
\end{proposition}

One of the main ingredients in the proof of \Cref{prop:nonexistenceSummableControl} is the method of sub and supersolutions: if \eqref{eq:nonlinearDirichletProblem} with datum \(\mu \in \cM(\Omega)\) has a subsolution \(\underline{u}\) and a supersolution \(\overline{u}\) such that \(\underline{u} \le \overline{u}\) almost everywhere in \(\Omega\), then \eqref{eq:nonlinearDirichletProblem} with datum \(\mu\) has a unique solution \(u\) which satisfies \(\underline{u} \le u \le \overline{u}\) almost everywhere in \(\Omega\); see~\cite{Ponce:2016}*{Proposition~20.5}. By a \emph{supersolution} of \eqref{eq:nonlinearDirichletProblem} with datum \(\mu \in \cM(\Omega)\), we mean a function \(\overline{u} \in W_0^{1, 1}(\Omega)\) such that \(g(\overline{u}) \in L^1(\Omega)\) and
\[
- \Delta \overline{u} + g(\overline{u}) \ge \mu \quad \text{in the sense of distributions in \(\Omega\).}
\]
Similarly, one also defines a \emph{subsolution} of \eqref{eq:nonlinearDirichletProblem}.

Another ingredient involved in the proof of \Cref{prop:nonexistenceSummableControl} is the connection between Bessel capacities and Hausdorff measures. More precisely, if \(K \subset \Omega\) is a compact set such that \(\cH^{N - d}(K) < \infty\), where \(2 < d \le N\), then \(\capt_{W^{2, s}}(K) = 0\) for every \(1 < s \le d/2\); see~\cite{AdamsHedberg:1996}*{Theorem~5.1.9}. The converse is not true, but if \(\capt_{W^{2, s}}(K) = 0\) for \(1 < s \le N / 2\), then \(\cH^\alpha(K) = 0\) for every \(\alpha > N - 2s\); see~\cite{AdamsHedberg:1996}*{Theorem~5.1.13}. By virtue of \eqref{eq:LaplacianSobolevCapacities}, the same properties hold with respect to the \(\DL{s}\) capacity.

\begin{proof}[Proof of \Cref{prop:nonexistenceSummableControl}]
Assume that \(q \ge \frac{N}{N - 2}\); the case \(q < \frac{N}{N - 2}\) will be explained afterwards. Let \(K \subset \Omega\) be a Cantor set such that
\begin{equation}
\label{eq:cantorHausdorffMeasure}
0 < \cH^{N - 2p'}(K) < \infty.
\end{equation}
Since \(p' < q'\), we have \(\capt_{\DL{p'}}(K; \Omega) = 0\) and \(\capt_{\DL{q'}}(K; \Omega) > 0\). Using the Riesz representation theorem and the Hahn--Banach theorem, one shows as in the proof of Proposition~A.17 in \cite{Ponce:2016} that there exists a nonnegative finite Borel measure \(\mu\) supported in \(K\) such that \(\mu(K) = 1\) and, for every nonnegative function \(\varphi \in C_c^\infty(\Omega)\),
\[
0 \le \int_\Omega {} \varphi \d\mu \le \C \norm{\Delta \varphi}_{L^{q'}(\Omega)}.
\]
This estimate implies that \(\mu\) is a diffuse measure with respect to the \(\DL{q'}\) capacity; see~\cite{PonceWilmet:2017}*{Proposition~3.1}. Since \(\mu\) has compact support, one also has, for every \(\zeta \in C_0^\infty(\overline{\Omega})\),
\begin{equation}
\label{eq:C0infEstimate}
\abs[\bigg]{\int_\Omega {} \zeta \d\mu} \le \C \norm{\Delta \zeta}_{L^{q'}(\Omega)}.
\end{equation}
Indeed, it suffices to apply the previous estimate with \(\varphi = \zeta \phi\), where \(\phi \in C_c^\infty(\Omega)\) is some fixed function such that \(\phi = 1\) in \(K\).

Let \(u\) be the unique solution of the linear Dirichlet problem
\[
\left\{
\begin{alignedat}{3}
- \Delta u &= \mu &&\quad \text{in \(\Omega\),} \\
u &= 0 &&\quad \text{on \(\partial\Omega\).}
\end{alignedat}
\right.
\]
By \eqref{eq:C0infEstimate} we have
\[
\abs[\bigg]{\int_\Omega {} u \Delta \zeta} \le \C \norm{\Delta \zeta}_{L^{q'}(\Omega)}.
\]
We then deduce from the Riesz representation theorem that \(u \in L^q(\Omega)\), and then \(g(u) \in L^1(\Omega)\). Since \(\mu\) is a nonnegative measure, the weak maximum principle implies that \(u \ge 0\) almost everywhere in \(\Omega\); see~\cite{Ponce:2016}*{Proposition~6.1}. Thus, \(u\) is a supersolution of \eqref{eq:nonlinearDirichletProblem} with datum \(\mu\). Since \(0\) is a subsolution of the same problem, it follows from the method of sub and supersolutions that \eqref{eq:nonlinearDirichletProblem} with datum \(\mu\) has a solution \(u_d\).

Let \(F\) be the cost functional associated to this desired state \(u_d\) and let \(\mu^*\) be an optimal control. In particular,
\[
F(\mu^*) \le F(\mu) = \alpha \norm{\mu}_{\cM(\Omega)} < \infty.
\]
By \Cref{prop:concentratedPart}, we then have
\[
(\mu^*)_\c = (- \Delta u_d)_\c = \mu_\c \neq 0;
\]
whence \(\mu^*\) is not diffuse with respect to the \(\DL{p'}\) capacity. In particular, \(\mu^* \not\in L^1(\Omega)\). The same argument shows that if \(\nu \in \cM(\Omega)\) is such that \(F(\nu) < \infty\), then \(\nu_\c = \mu_c \neq 0\). Hence, \(F \equiv \infty\) in \(L^1(\Omega)\). The proof of the proposition is complete when \(q \ge \frac{N}{N - 2}\).

In the case where \(1 \le q < \frac{N}{N - 2}\), every nonempty set has positive \(\DL{q'}\) capacity. Since \eqref{eq:nonlinearDirichletProblem} is solvable for every datum \(\mu \in \cM(\Omega)\), it thus suffices to take \(u_d\) as the unique solution of \eqref{eq:nonlinearDirichletProblem} corresponding to \(\mu = \cH^{N - 2p'} \lfloor_K\), where \(K\) is any compact subset of \(\Omega\) that satisfies \eqref{eq:cantorHausdorffMeasure}. The proof is thus complete.
\end{proof}

When \(p = \infty\) and \(g\) is an arbitrary nondecreasing continuous function, there always exists an ideal state \(u_d\) which depends on \(g\) such that \(\mu^* \not\in L^1(\Omega)\). More precisely, we have

\begin{proposition}
\label{prop:nonexistenceSummableControlInf}
Assume that \(p = \infty\). Then there exists \(u_d \in L^1(\Omega)\) such that the cost functional \(F\) with desired state \(u_d\) satisfies the conclusion of \Cref{prop:nonexistenceSummableControl}.
\end{proposition}

\begin{proof}
Let \(\mu \in \cM(\Omega)\) be a nonnegative good measure which is not diffuse with respect to the \(\DL{1}\) capacity; see~\cite{Ponce:2005}*{Theorem~1}. Let \(u_d\) be the unique solution of \eqref{eq:nonlinearDirichletProblem} corresponding to \(\mu\) and denote by \(F\) the cost functional with desired state \(u_d\). It follows from \Cref{prop:concentratedPart} that, for every \(\nu \in \cM(\Omega)\) such that \(F(\nu) < \infty\), we have
\[
\nu_\c = (- \Delta u_d)_\c = \mu_\c,
\]
where the subscript \(\c\) denotes the concentrated part of the measure with respect to the \(\DL{1}\) capacity. Since \(\mu\) is not diffuse, we have \(\nu_\c \not\equiv 0\). In particular, \(\nu \not\in L^1(\Omega)\). Hence, \(F \equiv \infty\) in \(L^1(\Omega)\). The proof is complete.
\end{proof}

%\begin{proof}
%By a classical result from potential theory, there exists a nonnegative good measure \(\mu_d \in \cM(\Omega)\) which is not diffuse with respect to the \(\DL{1}\) capacity; see e.g.~\cite{AdamsHedberg:1996}*{Theorem~2.5.1}. Let \(u_d\) be the unique solution of \eqref{eq:nonlinearDirichletProblem} corresponding to \(\mu_d\). It follows from \Cref{prop:concentratedPart} that
%\[
%(\mu^*)_\c = (- \Delta u_d)_\c = (\mu_d)_\c,
%\]
%where the subscript \(\c\) denotes the concentrated part of the measure with respect to the \(\DL{1}\) capacity. Since \(\mu_d\) is not diffuse, we have \((\mu^*)_\c \not\equiv 0\). In particular, \(\mu^* \not\in L^1(\Omega)\), which concludes the proof of the proposition.
%\end{proof}

\section{Lack of convexity and lower semicontinuity}
\label{sec:example}

In this section, we show that the cost functional \(F\) need not be convex or lower semicontinuous with respect to the total variation norm. For this purpose, we consider polynomial nonlinearities.

\begin{proposition}
\label{prop:nonconvexityF}
Assume that \(1 \le p < \infty\) and \(g(t) = \abs{t}^{p - 1} t\). Let \(u_d\) be the unique solution of \eqref{eq:nonlinearDirichletProblem} corresponding to some positive measure \(\mu \in \cM(\Omega)\) and let \(F\) be the cost functional with desired state \(u_d\). Then \(F\) is not convex.
\end{proposition}

By a positive measure, we mean that \(\mu \ge 0\) and \(\mu \not\equiv 0\). The proof of \Cref{prop:nonconvexityF} relies on the fact that the set of nonnegative good measures for \(g\) as above is a convex cone; see~\cite{BrezisMarcusPonce:2007}*{Proposition~4.3}.

\begin{proof}
Let \(\theta > 1\). To prove that \(F\) is not convex, if suffices to show that
\begin{equation}
\label{eq:nonconvexityF}
\frac{F(\mu) + F(\theta \mu)}{2} < F \paren[\bigg]{\frac{1 + \theta}{2} \mu}.
\end{equation}
On the one hand,
\begin{equation}
\label{eq:normLinearity}
\norm[\bigg]{\frac{1 + \theta}{2} \mu}_{\cM(\Omega)} = \frac{\norm{\mu}_{\cM(\Omega)} + \norm{\theta \mu}_{\cM(\Omega)}}{2}.
\end{equation}
On the other hand---denoting by \(v\) and \(w\) the solutions of \eqref{eq:nonlinearDirichletProblem} with data \(\frac{1 + \theta}{2} \mu\) and \(\theta \mu\), respectively---the convexity of \(g\) on \(\left[0, \infty\right[\) implies that \(\frac{u_d + w}{2}\) is a subsolution of \eqref{eq:nonlinearDirichletProblem} with datum \(\frac{1 + \theta}{2} \mu\). It then follows from the weak maximum principle that
\[
0 \le \frac{u_d + w}{2} \le v \quad \text{almost everywhere in \(\Omega\).}
\]
By strict convexity of \(g\), equality cannot hold almost everywhere. Since
\[
0 \le \frac{w - u_d}{2} \le v - u_d \quad \text{almost everywhere in \(\Omega\),}
\]
and equality fails on a set of positive measure, we have
\begin{equation}
\label{eq:normStrictInequality}
\frac{\norm{w - u_d}_{L^p(\Omega)}}{2} < \norm{v - u_d}_{L^p(\Omega)}.
\end{equation}
Combining \eqref{eq:normLinearity} and \eqref{eq:normStrictInequality} we obtain \eqref{eq:nonconvexityF}, from which the conclusion follows.
\end{proof}

We now develop the example given in the introduction. Namely, we prove

\begin{proposition}
\label{prop:notLowerSemicontinuous}
Assume that \(u_d \in L^p(\Omega)\) and \(g(t) = \abs{t}^{q - 1} t\) for some \(q \ge \frac{N}{N- 2}\). If~\(1 \le p \le q\), then \(F\) is not lower semicontinuous with respect to weak* convergence in \(\cM(\Omega)\).
\end{proposition}

\begin{proof}
Let \(a \in \Omega\), let \((\rho_n)_{n \in \N_*}\) be a sequence of translated mollifiers in \(C_c^\infty(\R^N)\) such that \(\supp \rho_n \subset B(a; 1/n)\) and let \(u_n\) be the unique solution of the Dirichlet problem
\[
\left\{
\begin{alignedat}{3}
- \Delta u_n + \abs{u_n}^{q - 1} u_n &= \rho_n &&\quad \text{in \(\Omega\),} \\
u_n &= 0 &&\quad \text{on \(\partial\Omega\).}
\end{alignedat}
\right.
\]
One shows that the sequence \((\rho_n)_{n \in \N_*}\) converges weakly* to the Dirac mass \(\delta_a\) in \(\cM(\Omega)\). However, in dimension \(N \ge 3\), the Dirichlet problem above with datum \(\delta_a\) has no solution if \(q \ge \frac{N}{N - 2}\); see~\cite{BenilanBrezis:2003}*{Remark~A.4}. Hence \(F(\delta_a) = \infty\). Our goal is to show that
\begin{equation}
\label{eq:limsupMollifiers}
\limsup_{n \to \infty} {} F(\rho_n) < \infty.
\end{equation}
On the one hand, it follows from \cite{Brezis:1983}*{Theorem~4} that the sequence \((u_n)_{n \in \N_*}\) converges strongly to \(0\) in \(L^1(\Omega)\). On the other hand, the absorption estimate \eqref{eq:absorptionEstimate} implies that
\[
\norm{u_n}_{L^q(\Omega)}^q = \norm{g(u_n)}_{L^1(\Omega)} \le \norm{\rho_n}_{L^1(\Omega)} \le 1.
\]
Since \(p \le q\) and \(\Omega\) is bounded, we deduce from H\"older's inequality that
\[
F(\rho_n) \le \norm{u_n}_{L^p(\Omega)} + \norm{u_d}_{L^p(\Omega)} + \alpha \norm{\rho_n}_{\cM(\Omega)} \le \abs{\Omega}^{\frac{1}{p} - \frac{1}{q}} + \norm{u_d}_{L^p(\Omega)} + \alpha.
\]
This yields \eqref{eq:limsupMollifiers} and the conclusion follows.
\end{proof}

Using the classical interpolation inequality in Lebesgue spaces, one can show that for \(p < q\),
\[
\lim_{n \to \infty} F(\rho_n) = \norm{u_d}_{L^p(\Omega)} + \alpha.
\]
The case \(p = q\) can be handled with the Brezis--Lieb lemma~\citelist{\cite{BrezisLieb:1983} \cite{Willem:2013}*{Theorem~4.2.7}}, which gives
\[
\lim_{n \to \infty} F(\rho_n) = \paren[\big]{\norm{u_d}_{L^p(\Omega)} + 1}^{1/p} + \alpha.
\]
In the case where \(p > q \ge \frac{N}{N - 2}\), the sequence \((u_n)_{n \in \N_*}\) cannot be bounded in \(L^p(\Omega)\). Indeed, if that were the case, then it would follow from the classical interpolation inequality in Lebesgue spaces that the sequence \((g(u_n))_{n \in \N_*}\) converges strongly to \(0\) in \(L^1(\Omega)\), contradicting the fact that such a sequence converges weakly* to \(\delta_a\) in \(\cM(\Omega)\).

\section{Remark on the Lavrentiev phenomenon}
\label{sec:lavrentiev}

In contrast with \Cref{prop:nonexistenceSummableControl} above, we show that the Lavrentiev phenomenon \emph{cannot} occur if \(g(t) = \abs{t}^{p - 1} t\) and \(u_d \in L^p(\Omega)\) for some \(1 \le p < \infty\). This is the content of

\begin{proposition}
\label{prop:noLavrentievPhenomenon}
Assume that \(1 \le p < \infty\) and \(g(t) = \abs{t}^{p - 1} t\). If \(u_d \in L^p(\Omega)\), then
\[
\inf_{\mu \in \cM(\Omega)} {} F(\mu) = \inf_{\mu \in L^1(\Omega)} {} F(\mu).
\]
\end{proposition}

In dimension \(N \ge 3\), we rely on a strong approximation property of diffuse measures based on the Hahn--Banach theorem~\cite{PonceWilmet:2017}*{Proposition~2.1}. The proof in dimension \(N = 2\) is simpler due to the fact that solutions of \eqref{eq:nonlinearDirichletProblem} are continuously embedded in \(L^q(\Omega)\) for every \(1 \le q < \infty\).

\resetconstant

\begin{proof}
We assume that \(N \ge 3\). The case \(N = 2\) will be explained afterward. Let \(\mu \in \cM(\Omega)\) be such that \(F(\mu) < \infty\) and let \(u\) be the unique solution of \eqref{eq:nonlinearDirichletProblem} corresponding to \(\mu\). We shall prove the existence of a sequence \((\mu_n)_{n \in \N}\) in \(L^1(\Omega)\) satisfying
\begin{equation}
\label{eq:approximationL1}
\lim_{n \to \infty} {} F(\mu_n) = F(\mu).
\end{equation}
For this purpose, we first assume that \(\mu\) is a nonnegative measure compactly supported in \(\Omega\) such that
\begin{equation}
\label{eq:trace}
\abs[\bigg]{\int_\Omega {} \zeta \d\mu} \le C \norm{\Delta \zeta}_{L^{p'}(\Omega)},
\end{equation}
for every \(\zeta \in C_0^\infty(\overline{\Omega})\). In this case, we show that the sequence \((\rho_n * \mu)_{n \in \N_*}\) satisfies \eqref{eq:approximationL1}. Let us denote by \(u_n\) the unique solution of \eqref{eq:nonlinearDirichletProblem} with datum \(\mu_n\). Taking a subsequence if necessary, we may assume that \((u_n)_{n \in \N*}\) converges strongly to some function \(v\) in \(L^1(\Omega)\). Since
\begin{equation}
\label{eq:approximationL1norm}
\lim_{n \to \infty} {} \norm{\rho_n * \mu}_{L^1(\Omega)} = \norm{\mu}_{\cM(\Omega)},
\end{equation}
we are left to prove that \((u_n)_{n \in \N_*}\) converges strongly to \(u\) in \(L^p(\Omega)\). To this end, we extend \(\mu\) by zero in \(\R^N \setminus \Omega\) and we let \(\NP\mu : \R^N \to \bracks{0, \infty}\) denote the Newtonian potential generated by \(\mu\): for every \(x \in \R^N\),
\[
\NP\mu(x) = \frac{1}{(N - 2) \sigma_N} \int_\Omega {} \frac{\d\mu(y)}{\abs{x - y}^{N - 2}},
\]
where \(\sigma_N\) is the surface measure of the unit sphere in \(\R^N\). One shows that \(\NP\mu\) belongs to \(L^1(\Omega)\) and satisfies the Poisson equation
\[
- \Delta \NP\mu = \mu \quad \text{in the sense of distributions in \(\R^N\);}
\]
see e.g.~\cite{Ponce:2016}*{Example~2.12}. We first claim that the convolution \(\rho_n * \NP\mu\) is a supersolution of \eqref{eq:nonlinearDirichletProblem} with datum \(\rho_n * \mu\). On the one hand, Fubini's theorem implies that
\[
- \Delta (\rho_n * \NP\mu) = \rho_n * \mu \quad \text{in \(\R^N\).}
\]
Since \(\rho_n * \NP\mu \ge 0\) on \(\partial \Omega\), by the Divergence theorem we have
\[
- \int_\Omega {} (\rho_n * v) \Delta \zeta \ge \int_\Omega {} \zeta (\rho_n * \mu),
\]
for every nonnegative function \(\zeta \in C_0^\infty(\overline{\Omega})\). Hence
\[
- \int_\Omega {} (\rho_n * \NP\mu) \Delta \zeta + \int_\Omega {} g(\rho_n * \NP\mu) \zeta \ge \int_\Omega {} \zeta (\rho_n * \mu).
\]
By the weak maximum principle, we thus have
\[
0 \le u_n \le \rho_n * \NP\mu \quad \text{almost everywhere in \(\Omega\).}
\]
We now claim that \(\NP\mu \in L^p(\Omega)\). For this purpose, let \(w\) be the unique solution of the Dirichlet problem
\[
\left\{
\begin{alignedat}{3}
- \Delta w &= \mu &&\quad \text{in \(\Omega\),} \\
w &= 0 &&\quad \text{on \(\partial\Omega\).}
\end{alignedat}
\right.
\]
Using the Riesz representation theorem, one deduces from estimate \eqref{eq:trace} that \(w \in L^p(\Omega)\). Since
\[
- \Delta (\NP\mu - w) = 0 \quad \text{in the sense of distributions in \(\Omega\),}
\]
the function \(\NP\mu - w\) is harmonic in \(\Omega\), which implies that \(\NP\mu \in L_\loc^p(\Omega)\). Since \(\NP\mu\) is harmonic in \(\R^N \setminus \supp \mu\), we thus have \(\NP\mu \in L^p(\Omega)\). Taking a further subsequence if needed, we may assume that the sequence \((u_n)_{n \in \N_*}\) converges almost everywhere to \(v\) in \(\Omega\). We then deduce from the Dominated convergence theorem that \((u_n)_{n \in \N_*}\) converges strongly to \(v\) in \(L^p(\Omega)\). Since \((g(u_n))_{n \in \N_*}\) converges strongly to \(g(v)\) is \(L^1(\Omega)\), it follows by uniqueness that \(v = u\).

We now assume that \(\mu\) is a signed measure such that \(\abs{\mu}\) satisfies \eqref{eq:trace}. Applying the previous case to the positive and negative parts of \(\mu\), one shows that \(\NP\abs{\mu} \in L^p(\Omega)\) and
\[
0 \le \abs{u_n} \le \rho_n * \NP\abs{\mu} \quad \text{almost everywhere in \(\Omega\).}
\]
Then, passing to a subsequence if necessary, \eqref{eq:approximationL1} follows from \eqref{eq:approximationL1norm} and the Dominated convergence theorem.

In the case of an arbritrary measure \(\mu \in \cM(\Omega)\) such that \(F(\mu) < \infty\), we rely on a characterization of good measures for polynomial nonlinearities due to Baras and Pierre~\cite{BarasPierre:1984}. More precisely, we have that good measures must be diffuse with respect to the \(\DL{p'}\) capacity; see~\cite{BarasPierre:1984}*{Th\'eor\`eme~4.1}. In particular, \(\mu\) is diffuse and there exists a sequence \((\mu_n)_{n \in \N}\) of compactly supported measures in \(\cM(\Omega)\) such that \(\abs{\mu_n}\) satisfies \eqref{eq:trace} and
\[
\lim_{n \to \infty} {} \norm{\mu_n - \mu}_{\cM(\Omega)} = 0;
\]
see~\cite{PonceWilmet:2017}*{Proposition~2.1}. Let \(u_n\) denote the unique solution of \eqref{eq:nonlinearDirichletProblem} with datum \(\mu_n\). One shows that the sequences \((u_n)_{n \in \N}\) and \((g(u_n))_{n \in \N}\) converge strongly in \(L^1(\Omega)\) to the functions \(u\) and \(g(u)\), respectively; see~\cite{BrezisMarcusPonce:2007}*{Proposition~4.2}. Thus, \((u_n)_{n \in \N}\) converges strongly to \(u\) in \(L^p(\Omega)\), which implies that
\[
\lim_{n \to \infty} {} F(\mu_n) = F(\mu).
\]
For every \(j \in \N_*\) and every \(k \in \N\), we have
\[
\abs{F(\rho_j * \mu_k) - F(\mu)} \le \abs{F(\rho_j * \mu_k) - F(\mu_k)} + \abs{F(u_k) - F(\mu)}.
\]
Let \((n_k)_{k \in \N}\) be an increasing sequence of indices such that, for every \(k \in \N_*\),
\[
\abs{F(\rho_{n_k} * \mu_k) - F(\mu_k)} \le 1/k.
\]
Then the sequence \((\rho_{n_k} * \mu_k)_{k \in \N_*}\) is contained in \(L^1(\Omega)\) and satisfies \eqref{eq:approximationL1}. The proof is complete when \(N \ge 3\).

In the case where \(N = 2\), the solution \(u\) of \eqref{eq:nonlinearDirichletProblem} corresponding to \(\mu \in \cM(\Omega)\) belongs to \(L^q(\Omega)\) for every \(1 \le q < \infty\) and satisfies the estimate
\[
\norm{v}_{L^q(\Omega)} \le C \norm{\nu}_{\cM(\Omega)},
\]
for some constant \(C > 0\) depending on \(q\) and \(\Omega\). Hence, letting \(u_n\) denote the unique solution of \eqref{eq:nonlinearDirichletProblem} with datum \(\rho_n * \mu\), taking a subsequence if necessary, we directly deduced from the interpolation inequality in Lebesgue spaces and from the uniqueness of \(u\) that \((u_n)_{n \in \N_*}\) converges strongly to \(u\) in \(L^p(\Omega)\). From this, \eqref{eq:approximationL1} follows and the proof of the proposition when \(N = 2\) is complete.
\end{proof}

\end{document}